\documentclass[draft]{svproc}
\usepackage{url}

\usepackage{amsmath,amssymb,bm,cite}

\def\BC{\mathbb C}
\def\BN{\mathbb N}
\def\BR{\mathbb R}

\def\cD{\mathcal D}

\def\rd{\mathrm d}

\def\Ga{\Gamma}

\def\Om{\Omega}
\def\al{\alpha}
\def\be{\beta}
\def\ga{\gamma}
\def\de{\delta}

\def\te{\theta}

\def\ka{\kappa}
\def\la{\lambda}

\def\vp{\varphi}
\def\om{\omega}
\def\f{\frac}

\def\ov{\overline}
\def\pa{\partial}
\def\wh{\widehat}
\def\wt{\widetilde}

\allowdisplaybreaks[4]

\begin{document}
\mainmatter              
\title{Uniqueness of inverse source problems for\\
time-fractional diffusion equations with\\
singular functions in time}
\titlerunning{Inverse problems for fractional equations with singular time functions}  
%
\author{Yikan Liu\inst{1} \and Masahiro Yamamoto\inst{2,3,4}}
\authorrunning{Yikan Liu et al.} 
%
\tocauthor{Yikan Liu, and Masahiro Yamamoto}
\institute{Research Center of Mathematics for Social Creativity, Research Institute for Electronic Science, Hokkaido University, N12W7, Kita-Ward, Sapporo 060-0812, Japan,\\
\email{ykliu@es.hokudai.ac.jp}
\and
Graduate School of Mathematical Sciences, The University of Tokyo, 3-8-1 Komaba, Meguro-ku, Tokyo 153-8914, Japan,\\
\email{myama@next.odn.ne.jp}
\and
Honorary Member of Academy of Romanian Scientists, Ilfov, nr. 3, Bucuresti, Romania
\and
Correspondence member of Accademia Peloritana dei Pericolanti, Palazzo Universit\`a, Piazza S. Pugliatti 1 98122 Messina, Italy
}

\maketitle              

\begin{abstract}
We consider a fractional diffusion equations of order $\al\in(0,1)$ whose source term is singular in time:
\[
(\pa_t^\al+A)u(\bm x,t)=\mu(t)f(\bm x),\quad(\bm x,t)\in\Om\times(0,T),
\]
where $\mu$ belongs to a Sobolev space of negative order. In inverse source problems of determining $f|_\Om$ by the data $u|_{\om\times(0,T)}$ with a given subdomain $\om\subset\Om$ or $\mu|_{(0,T)}$ by the data $u|_{\{\bm x_0\}\times(0,T)}$ with a given point $\bm x_0\in\Om$, we prove the uniqueness by reducing to the case  $\mu\in L^2(0,T)$. The key is a transformation of a solution to an initial-boundary value problem with a regular function in time.

\end{abstract}


\section{Introduction}\label{sec-intro}

Let $\Om\subset\BR^d$ ($d\in\BN:=\{1,2,\ldots\}$) be a bounded domain with a smooth boundary $\pa\Om$, and let $\bm\nu=\bm\nu(\bm x)$ be the unit outward normal vector of $\pa\Om$. We set 
\begin{equation}\label{eq1.1}
A v(\bm x)=-\sum_{i,j=1}^d\pa_j(a_{i j}(\bm x)\pa_i v(\bm x))+\sum_{j=1}^d b_j(\bm x)\pa_j v(\bm x)+c(\bm x)v(\bm x),\quad\bm x\in\Om,
\end{equation}
where $a_{i j}=a_{j i}\in C^1(\ov\Om)$, $b_j\in C(\ov\Om)$ ($1\le i,j\le d$), $c\in C(\ov\Om)$  and we assume that there exists a constant $\ka>0$ such that 
\[
\sum_{i,j=1}^d a_{i j}(\bm x)\xi_i\xi_j\ge\ka\sum_{j=1}^d|\xi_j|^2,\quad\forall\,\bm x\in\Om,\ \forall\,(\xi_1,\ldots,\xi_d)\in\BR^d.
\]
We consider an initial-boundary value problem for a time-fractional diffusion equation whose source term is described by $\mu(t)f(\bm x)$, where $f$ is a spatial distribution of the source and $\mu$ is a temporal change factor. We can describe a governing initial-boundary value problem as follows:
\begin{equation}\label{eq1.2}
\begin{cases}
(\rd_t^\al+A)u(\bm x,t)=\mu(t)f(\bm x), & (\bm x,t)\in\Om\times(0,T),\\
u(\bm x,0)=0, & \bm x\in\Om,\\
u(\bm x,t)=0, & (\bm x,t)\in\pa\Om\times(0,T).
\end{cases}
\end{equation}
Here, for $0<\al<1$, we can formally define the pointwise Caputo derivative as
\[
\rd_t^\al v(t):=\f1{\Ga(1-\al)}\int^t_0 (t-s)^{-\al}v'(s)\,\rd s,\quad v\in W^{1,1}(0,T).
\]

Owing to their capability of describing memory effects, time-fractional partial differential equations such as \eqref{eq1.2} have gathered consistent popularity among multidisciplinary researchers as models for anomalous diffusion and viscoelasticity (e.g. \cite{BP88,BDES18,HH98}). Though the history of fractional calculus can be traced back to Leibniz, the main focus on fractional equations was biased to the construction of explicit and approximate solutions via special functions and transforms until the last decades due to the needs from applied science. It has only been started recently that problems like \eqref{eq1.2} are formulated in appropriate function spaces using modern mathematical tools, followed by rapidly increasing literature on their fundamental theories, numerical analysis and inverse problems. Here we do not intend to give a complete bibliography, but only refer to several milestone works \cite{EK04,GLY,KRY,SY} and the references therein.

Especially, the source term in \eqref{eq1.2} takes the form of separated variables, where $\mu(t)$ and $f(\bm x)$ describe the time evolution and the spatial distribution of some contaminant source, respectively. Therefore, the determination of $\mu(t)$ or $f(\bm x)$ turns out to be important in the context of environmental issues, which motivates us to propose the following problem.

\begin{problem}[inverse source problems]\label{isp}
Let $u$ satisfy \eqref{eq1.2}, $\emptyset\ne\om\subset\Om$ be an arbitrary subdomain and $\bm x_0\in\Om$ be an arbitrary point. Determine $\mu|_{(0,T)}$ by $u|_{\{\bm x_0\}\times(0,T)}$ with given $f$ and $f|_\Om$ by $u|_{\om\times (0,T)}$ with given $\mu,$ respectively.
\end{problem}

Indeed, Problem~\ref{isp} includes two inverse problems, namely, the determination of $\mu(t)$ by the single point observation and that of $f(\bm x)$ by the partial interior observation of $u$. Both problems have been studied intensively in the last decade, especially among which the uniqueness was already known in literature. We can refer to many works, but here only to \cite{LRY,L17,LZ17} for the inverse $t$-source problem and \cite{JLLY,KSXY,KLY,LLY22} for the inverse $\bm x$-source problem. For a comprehensive survey on Problem~\ref{isp} especially before 2019, we refer to Liu, Li and Yamamoto \cite{LLY}.

However, it reveals that the existing papers mainly discuss regular temporal components $\mu$, e.g. in $C^1[0,T]$ or $L^2(0,T)$. Such restrictions exclude a wide class of singular functions represented by the Dirac delta function, which corresponds with point sources in practice. This encourages us to reconsider Problem~\ref{isp} in function spaces with lower regularity. More precisely, in this article we are mainly concerned with the case of $\mu=\mu(t)$ in a Sobolev space of negative order, in particular $\mu\not\in L^2(0,T)$.

For such less regular $\mu$, we cannot expect the differentiability of $u(\bm x,\,\cdot\,)$ in time and we must redefine the Caputo derivative $\rd_t^\al$ in \eqref{eq1.2}. In this case, the unique existence of a solution to the initial-boundary value problem is more delicate, and an adequate formulation is needed.

This article is composed of 6 sections and 1 appendix. Redefining $\rd_t^\al$ and thus problem \eqref{eq1.2} in negative Sobolev spaces, in Sect.~\ref{sec-premain} we state the main result in this paper. Then Sect.~\ref{sec-trans}--\ref{sec-Duhamel} are devoted to the two key ingredients for treating the singular system, i.e., the transfer to another regular system in $L^2$ and Duhamel's principle in $_\al H(0,T)$-space. Next, the proof of Theorem~\ref{thm1} is completed in Sect.~\ref{sec-proof}. Finally, Sect.~\ref{sec-remark} provides concluding remarks, and the proof of a technical detail is postponed to Appendix~\ref{sec-app}.


\section{Preliminary and Statement of the Main Result}\label{sec-premain}

To start with, we first define a fractional derivative for $v\in L^2(0,T)$ which extends the domain of  $\rd_t^\al$ and formulate the initial-boundary value problem. To this end, we introduce function spaces and operators. Set the forward and backward Riemann-Liouville integral operators as
\begin{align*}
(J_\al v)(t) & :=\f1{\Ga(\al)}\int^t_0(t-s)^{\al-1}v(s)\,\rd s,\quad0<t<T,\quad\cD(J_\al)=L^2(0,T),\\
(J^\al v)(t) & :=\f1{\Ga(\al)}\int^T_t(s-t)^{\al-1}v(s)\,\rd s,\quad0<t<T,\quad\cD(J_\al)=L^2(0,T).
\end{align*}
We define an operator $\tau:L^2(0,T)\longrightarrow L^2(0,T)$ by $(\tau v)(t):=v(T-t)$ and obviously $\tau$ is an isomorphism. We set
\begin{align*}
{}_0C^1[0,T] & :=\{v\in C^1[0,T]\mid v(0)=0\},\\
{}^0C^1[0,T] & :=\{v\in C^1[0,T]\mid v(T)=0\}=\tau({}_0C^1[0,T]).
\end{align*}
By $H^\al(0,T)$ we denote the Sobolev-Slobodecki space with the norm $\|\cdot\|_{H^\al(0,T)}$ defined by
\[
\|v\|_{H^\al(0,T)}:=\left(\| v\|^2_{L^2(0,T)}+\int^T_0\!\!\!\int^T_0 \f{|v(t)-v(s)|^2}{|t-s|^{1+2\al}}\,\rd t\rd s\right)^{1/2}
\]
(e.g., Adams \cite{Ad}). Then we further introduce the function spaces
\[
{}_\al H(0,T):=\ov{{}_0C^1[0,T]}^{H^\al(0,T)},\quad {}^\al H(0,T):=\ov{^{0}{C^1}[0,T]}^{H^\al(0,T)}.
\]
Regarding the operators $J_\al,J^\al$ and the spaces ${}_\al H(0,T),{}^\al H(0,T)$, we have the following lemma.

\begin{lemma}\label{lem1.1}
Let $0<\al<1$.

{\rm(i)} $J_\al:L^2(0,T)\longrightarrow{}_\al H(0,T)$ is bijective and isomorphism.

{\rm(ii)} $J^\al:L^2(0,T)\longrightarrow{}^\al H(0,T)$ is bijective and isomorphism.
\end{lemma}
	
As for the proof of Lemma~\ref{lem1.1}(i), see Gorenflo, Luchko and Yamamoto \cite{GLY}, Kubica, Ryszewska and Yamamoto \cite{KRY}. Lemma~\ref{lem1.1}(ii) can be readily derived from Lemma~\ref{lem1.1}(i), the identity $J^\al v=\tau(J_\al(\tau v))$ and the fact that $\tau$ is an isomorphism.

Based on the spaces ${}_\al H(0,T)$ and ${}^\al H(0,T)$, let
\begin{align*}
& {}_\al H(0,T)\subset L^2(0,T)\subset({}_\al H(0,T))'=:{}_{-\al}H(0,T),\\
& {}^\al H(0,T)\subset L^2(0,T)\subset({}^\al H(0,T))'=:{}^{-\al}H(0,T)
\end{align*}
be the Gel'fand triples, where $(\,\cdot\,)'$ denotes the dual space. Let $X,Y$ be Hilbert spaces and $K:X\longrightarrow Y$ be a bounded linear operator with its domain $\cD(K)=X$. Then we recall that $K'$ is the maximal operator among $\wh K:Y'\longrightarrow X'$ with $\cD(\wh K)\subset Y'$ such that
\[ 
{}_{X'}\langle \wh K y,x\rangle_X={}_{Y'}\langle y,K x\rangle_Y,\quad\forall\,x\in X,\ \forall\,y\in\cD(\wh K)\subset Y'.
\]
Then we can prove the following lemma.

\begin{lemma}[{Yamamoto \cite[Proposition 9]{Y22}}]\label{lem1.2}
Let $0<\be<1$.  

{\rm(i)} $(J_\be)':{}_{-\be}H(0,T)\longrightarrow L^2(0,T)$ is bijective and isomorphism.

{\rm(ii)} $(J^\be)':{}^{-\be}H(0,T)\longrightarrow L^2(0,T)$ is bijective and isomorphism.

{\rm(iii)} There holds $J_\be\subset(J^\be)'$. In particular, $J_\be v=(J^\be)'v$ for $v\in L^2(0,T)$.
\end{lemma}

Now we can redefine the Caputo derivative for functions in $L^2(0,T)$.

\begin{definition}
We define 
\[
\pa_t^\al :=((J^\al)')^{-1},\quad\cD(\pa_t^\al)=L^2(0,T).
\]
\end{definition}

\begin{remark}
Kubica, Ryszewska and Yamamoto \cite{KRY} defined $\wt\pa_t^\al$ by
\[
\wt\pa_t^\al:=(J_\al)^{-1},\quad\cD(\wt\pa_t^\al)={}_\al H(0,T).
\]
By Lemma~\ref{lem1.2}(iii), we see $\pa_t^\al\supset\wt\pa_t^\al$ and
\begin{equation}\label{eq1.3}
\wt\pa_t^\al v=\pa_t^\al v=\lim_{n\to\infty}\rd_t^\al v_n\quad\mbox{in }L^2(0,T),\ \forall\,v\in{}_\al H(0,T),
\end{equation}
where $v_n\in{}_0C^1[0,T]$ and $v_n\longrightarrow v$ in $H^\al(0,T)$. Thus $\pa_t^\al$ is an extension of $\wt\pa_t^\al$.
\end{remark}

Now we are ready to propose the initial-boundary value problem:
\begin{align}
& (\pa_t^\al+A)u=\mu(t)f(\bm x)\quad\mbox{in }^{-\be}H(0,T;L^2(\Om)),\label{eq1.4}\\
& u\in L^2(0,T;L^2(\Om))\cap{}^{-\be}{H(0,T;H^2(\Om)\cap H^1_0(\Om))},\label{eq1.5}
\end{align}
where
\begin{equation}\label{eq1.6}
\mu\in{}^{-\be}H(0,T),\quad f\in L^2(\Om),\quad\al\le\be<1,\quad\be>1/2.
\end{equation}
Henceforth we assume the existence of such a solution $u$, although it can be verified during the succeeding arguments.

Now we are well prepared to state the main result on the uniqueness for Problem~\ref{isp} with singularity in time.

\begin{theorem}\label{thm1}
Let $u$ satisfy {\rm\eqref{eq1.4}--\eqref{eq1.5}} and assume \eqref{eq1.6}. Let $\om\subset\Om$ and $\bm x_0\in\Om$ be the same as that in Problem~$\ref{isp}$.

{\rm(i)} Let $\mu\not\equiv0$ in $(0,T)$. If $u=0$ in $\om\times (0,T),$ then $f=0$ in $L^2(\Om)$.

{\rm(ii)} In \eqref{eq1.1} we assume that $b_j=0\ (j=1,\ldots,d)$ and $c\ge0$ in 
$\Om$. Let $f$ satisfy
\begin{gather}
f\not\equiv0\mbox{ in }\Om\quad\mbox{and}\quad(f\ge0\mbox{ or }f\le0\mbox{ in }\Om),\label{eq1.8}\\
f\in\begin{cases}
L^2(\Om) & \mbox{if }d=1,2,3,\\
\cD(A^\te) & \mbox{if }d\ge4,\ \mbox{where }\te>d/4-1.
\end{cases}\label{eq1.9}
\end{gather}
Then $u=0$ at $\{\bm x_0\}\times(0,T)$ implies $\mu=0$ in $^{-\be}H(0,T)$.
\end{theorem}

In Theorem~\ref{thm1}(ii), in terms of \eqref{eq1.9} we can verify
\begin{equation}\label{eq1.10}
u\in{}^{-\be}H(0,T;C(\ov\Om)).
\end{equation}
Therefore, the data $u(\bm x_0,\,\cdot\,)\in{}^{-\be}H(0,T)$ makes sense. The proof of \eqref{eq1.10} is given in Appendix~\ref{sec-app}.

In Theorem~\ref{thm1}(ii), we can further study and prove the uniqueness in the case where we replace assumption \eqref{eq1.8} by more general conditions, e.g., $b_j\ne0$ and $c$ is not necessarily non-negative, but we omit details. Moreover, here we omit the treatments for the case $1<\al<2$.

Since $^{\be}H(0,T)\subset C[0,T]$ by $\be>1/2$, we see that the following $\mu$ are in $^{-\be}H(0,T)$: $\mu\in L^1(0,T)$ and $\mu(t)=\sum_{k=1}^N r_k\de_{a_k}(t)$, where $r_k\in\BR$, $0<a_k<T$ and $\de_{a_k}$ is the Dirac delta function: $\langle\de_{a_k},\vp\rangle=\vp(a_k)$ for $\vp\in C[0,T]$. In particular, from Theorem~\ref{thm1}(ii) we can directly derive

\begin{corollary}
Under the same conditions as that in Theorem {\rm\ref{thm1}(ii),} we set 
\[
\mu^\ell(t)=\sum_{k=1}^{N^\ell}r_k^\ell\de_{a_k^\ell}(t),\quad\ell=1,2,\quad a_k^\ell\in(0,T),\quad r_k^\ell\in\BR\setminus\{0\},\quad k=1,\ldots,N^\ell,
\]
where $a_k^\ell\ (k=1,2,\ldots,N^\ell)$ are mutually distinct for $\ell=1,2$. Let $u^\ell$ be the solution to {\rm\eqref{eq1.4}--\eqref{eq1.5}} with $\mu=\mu^\ell\ (\ell=1,2)$. Then $u^1(\bm x_0,t)=u^2(\bm x_0,t)$ for $0<t<T$ implies $N^1=N^2,$ $r_k^1=r_k^2$ and $a_k^1=a_k^2\ (k=1,\ldots,N^1)$.
\end{corollary}

For the case of $\mu\in L^2(0,T)$ or $\mu\in L^1(0,T)$, there are rich references but we are here restricted to \cite{JLLY,KLY,KSXY,LRY}. In particular, \cite{KLY} considers also the case $0<\al\le2$ and variable $\al(\bm x)$ for $\mu\in L^1(0,T)$, and we can apply our method to their formulated inverse problems, as mentioned in Sect.~\ref{sec-remark}.


\section{Transfer to a Regular System}\label{sec-trans}

The main purpose of this section is the proof of

\begin{proposition}\label{prop2.1}
Let $u$ satisfy {\rm\eqref{eq1.4}--\eqref{eq1.5}} with \eqref{eq1.6}. Then 
\begin{equation}\label{eq2.1}
v:=(J^\be)'u\in{}_{\be}H(0,T;L^2(\Om))\cap L^2(0,T;H^2(\Om)\cap H^1_0(\Om))
\end{equation}
satisfies
\begin{equation}\label{eq2.2}
(\pa_t^\al+A)v=(J^\be)'\mu\,f\quad\mbox{in }L^2(0,T;L^2(\Om)).
\end{equation}
\end{proposition}

\begin{proof}
We show 

\begin{lemma}\label{lem2.1}
Let $\al,\ga>0$ and $\al+\ga<1$. Then
\[
(J^{\al+\ga})'v=(J^\al)'(J^\ga)'v=(J^\ga)'(J^\al)'v\quad\mbox{for }v\in{}^{-\al-\ga}H(0,T).
\]
\end{lemma}

\begin{proof}
We can directly verify 
\[
J^{\al+\ga}=J^\al J^\ga=J^\ga J^\al\quad\mbox{in }L^2(0,T).
\]
Therefore, by e.g. Kato \cite[Problem 5.26 (p.168)]{Ka}, we conclude the lemma.\qed
\end{proof}

We complete the proof of Proposition~\ref{prop2.1}. Since $u\in L^2(0,T;L^2(\Om))$, by Lemma~\ref{lem1.2} and $\al\le\be$, we see
\[
(J^\be)'u=J_\be u\in{}_\be H(0,T;L^2(\Om))\subset{}_\al H(0,T;L^2(\Om)).
\]
Since $u\in{}^{-\al}H(0,T;H^2(\Om)\cap H_0^1(\Om))$, Lemma~\ref{lem1.2} yields $(J^\al)'u\in L^2(0,T;H^2(\Om)\cap H_0^1(\Om))$. Then
\begin{align*}
(J^\be)'u=(J^{\be-\al})'(J^\al)'u=J_{\be-\al}(J^\al)'u & \in{}_{\be-\al}H(0,T;H^2(\Om)\cap H_0^1(\Om))\\
& \subset L^2(0,T;H^2(\Om)\cap H_0^1(\Om))
\end{align*}
by Lemma~\ref{lem1.1}(i), Lemma~\ref{lem1.2}(iii) and $(J^\al)'u\in L^2(0,T;H^2(\Om)\cap H_0^1(\Om))$. Therefore, we see $v\in L^2(0,T;H^2(\Om)\cap H_0^1(\Om))$.

Next we operate $(J^\be)'$ on both sides of \eqref{eq1.4} to deduce
\[
(J^\be)'((J^\al)')^{-1}u+A(J^\be)'u=(J^\be)'\mu\,f.
\]
We will prove 
\[
(J^\be)'((J^\al)')^{-1}u=((J^\al)')^{-1}(J^\be)'u,
\]
that is,
\[
(J^\al)'(J^\be)'((J^\al)')^{-1}u=(J^\be)'u. 
\]
In fact, by Lemma~\ref{lem2.1}, we see $(J^\be)'=(J^{\be-\al})'(J^\al)'=(J^\al)'(J^{\be-\al})'=(J^\be)'$ and then
\begin{align*}
(J^\al)'(J^\be)'((J^\al)')^{-1}u& =(J^\al)'(J^{\be-\al})'(J^\al)'((J^\al)')^{-1}u=(J^\al)'(J^{\be-\al})'u\\
& =(J^\be)'u.
\end{align*}
Therefore, it follows that
\[
(J^\be)'((J^\al)')^{-1}u=((J^\al)')^{-1}(J^\be)'u=((J^\al)')^{-1}v=\pa_t^\al v.
\]
Hence, we conclude $\pa_t^\al v+A v=(J^\be)'\mu\,f$.\qed
\end{proof}

In terms of Proposition~\ref{prop2.1}, the unique existence of the solution $u$ to \eqref{eq1.4}--\eqref{eq1.5} can be clarified via $v:=(J^\be)'u$ in \eqref{eq2.1}.


\section{Duhamel's Principle}\label{sec-Duhamel}

In this section, we establish Duhamel's principle which transforms a solution to an initial-boundary value problem without the inhomogeneous term to a solution to \eqref{eq2.2}. Such a principle is known and see e.g. \cite{LRY} and we can refer also to the survey \cite{LLY} and Umarov \cite{Um}. Here we reformulate the formula in the space $_\al H(0,T)$ in order to apply within our framework.

\begin{lemma}[Duhamel's principle in $_\al H(0,T)$]\label{lem3.1}
Let $g\in L^2(0,T)$ and $f\in L^2(\Om)$. Let $z$ satisfy
\begin{equation}\label{eq3.1}
\begin{cases}
\pa_t^\al (z-f)+Az=0,\\
z-f\in{}_\al H(0,T;L^2(\Om)),\quad z\in L^2(0,T;H^2(\Om)\cap H^1_0(\Om)).
\end{cases}
\end{equation}
Then
\begin{equation}\label{eq3.2}
w(\,\cdot\,,t)=\int^t_0g(s)z(\,\cdot\,,t-s)\,\rd s,\quad0<t<T
\end{equation}
satisfies
\begin{equation}\label{eq3.3}
\begin{cases}
(\pa_t^\al+A)w=J_{1-\al}g\,f\quad\mbox{in }L^2(0,T;L^2(\Om)),\\
w\in{}_\al H(0,T;L^2(\Om))\cap L^2(0,T;H^2(\Om)\cap H^1_0(\Om)).
\end{cases}
\end{equation}
\end{lemma}

\begin{proof}
{\it Step 1.} We prove

\begin{lemma}\label{lem3.2}
Let $g\in L^2(0,T)$ and $v\in{}_\al H(0,T)$. Then
\begin{equation}\label{eq3.4}
\pa_t^\al\left(\int^t_0g(s)v(t-s)\,\rd s\right)=\int^t_0g(s)\pa_t^\al v(t-s)\,\rd s\quad\mbox{in }L^2(0,T).
\end{equation}
\end{lemma}

\begin{proof}
First we assume $g\in C_0^1[0,T]:=\{h\in C^1[0,T]\mid h(0)=h(T)=0\}$ and $v\in{}_0C^1[0,T]$. Then by $v(0)=0$, we see
\[
\left(\int^t_0g(s)v(t-s)\,\rd s\right)'=\int^t_0g(s)v'(t-s)\,\rd s\in L^1(0,T)
\]
and thus
\begin{align*}
\pa_t^\al\left(\int^t_0g(s)v(t-s)\,\rd s\right) & =\rd_t^\al\left(\int^t_0g(s)v(t-s)\,\rd s\right)\\
& =\f1{\Ga(1-\al)}\int^t_0 (t-s)^{-\al}\left(\int^s_0g(r)v'(s-r)\,\rd r\right)\rd s\\
& =\int^t_0g(r)\left(\f1{\Ga(1-\al)}\int^t_r(t-s)^{-\al}v'(s-r)\,\rd s\right)\rd r\\
& =\int^t_0g(r)\,\rd_t^\al v(t-r)\,\rd r=\int^t_0g(r)\pa_t^\al v(t-r)\,\rd r,
\end{align*}
where we exchanged the orders of the integrals with respect to $s$ and $r$. Therefore, \eqref{eq3.4} holds for each $g\in C_0^1[0,T]$ and $v\in{}_0C^1[0,T]$.  

Next, let $g\in L^2(0,T)$ and $v\in{}_\al H(0,T)$. Since $_\al H(0,T)=\ov{_0C^1[0,T]}^{H^\al(0,T)}$ (e.g., \cite{KRY}) and $L^2(0,T)=\ov{C_0^1[0,T]}^{L^2(0,T)}$, we can choose sequences $\{g_n\}\subset C_0^1[0,T]$ and $\{v_n\}\subset{}_0C^1[0,T]$ such that $g_n\longrightarrow g$ in $L^2(0,T)$ and $v_n\longrightarrow v$ in $_\al H(0,T)$ as $n\to\infty$. Henceforth we write $(g*v)(t)=\int^t_0g(s)v(t-s)\,\rd s$ for $0<t<T$, and we regard $\pa_t^\al$ as $(J_\al)^{-1}$, that is, $\cD(\pa_t^\al)={}_\al H(0,T)$ (see also \eqref{eq1.3}). As is directly proved,
we have $g_n*v_n\in{}_0C^1[0,T]$ and so $g_n*v_n\in{}_\al H(0,T)=\cD(\pa_t^\al)$. Therefore, we see
\[
\pa_t^\al(g_n*v_n)(t)=(g_n*\pa_t^\al v_n)(t),\quad0<t<T.
\]
Since $v_n\longrightarrow v$ in $_\al H(0,T)$, it follows that $\pa_t^\al v_n\longrightarrow\pa_t^\al v$ in $L^2(0,T)$. Then Young's convolution inequality yields $g_n*\pa_t^\al v_n\longrightarrow g*\pa_t^\al v$ in $L^2(0,T)$. Therefore, $\pa_t^\al(g_n*v_n)$ converges in $L^2(0,T)$ and 
$\pa_t^\al(g*v)=g*\pa_t^\al v$ in $L^2(0,T)$. The proof of Lemma~\ref{lem3.2} is complete.\qed
\end{proof}

{\it Step 2.} First we have
\[
w(\,\cdot\,,t)=\int^t_0g(s)(z(\,\cdot\,,t-s)-f)\,\rd s+f\int^t_0g(s)\,\rd s,\quad0<t<T.
\]
Since $\int^t_0g(s)\,\rd s=(J_1g)(t)$ for $0<t<T$ and $(\pa_t^\al J_1g)(t)=(J_\al^{-1}J_1g)(t)=(J_{1-\al}g)(t)$, by Lemma~\ref{lem3.2} and $z-f\in{}_\al H(0,T;L^2(\Om))$ we see
\[
\pa_t^\al w(\,\cdot\,,t)=\int^t_0g(s)\pa_t^\al(z-f)(\,\cdot\,,t-s)\,\rd s+(J_{1-\al}g)(t)f,\quad0<t<T.
\]
Moreover, we have $(A w)(\,\cdot\,,t)=\int^t_0g(s)A z(\,\cdot\,,t-s)\,\rd s$. Hence, we arrive at
\begin{align*}
(\pa_t^\al+A)w(\,\cdot\,,t) & =\int^t_0g(s)(\pa_t^\al(z-f)+A z)(\,\cdot\,,t-s)\,\rd s+(J_{1-\al}g)(t)f\\
& =(J_{1-\al}g)(t)f,\quad0<t<T.
\end{align*}
The regularity of $w$ described in \eqref{eq3.3} can follow directly from 
\eqref{eq3.2}. Thus, by the uniqueness of solution to \eqref{eq3.3}, the proof of Lemma~\ref{lem3.1} is complete.\qed
\end{proof}


\section{Completion of the Proof of Theorem~\ref{thm1}}\label{sec-proof}

We first show the following key lemma.

\begin{lemma}\label{lem4.1}
We assume
\[
\begin{cases}
(\pa_t^\al+A)v=g\,f\quad\mbox{in }L^2(0,T;L^2(\Om)),\\
v\in{}_\al H(0,T;L^2(\Om))\cap L^2(0,T;H^2(\Om)\cap H^1_0(\Om)), 
\end{cases}
\]
where $g\in L^2(0,T)$ and $f\in L^2(\Om)$. Let $\om\subset\Om$ and $\bm x_0\in\Om$ be the same as that in Problem~$\ref{isp}$.

{\rm(i)} Let $g\not\equiv0$ in $(0,T)$. If $v=0$ in $\om\times(0,T),$ then $f\equiv0$ in $\Om$.

{\rm(ii)} Let $f$ satisfy {\rm\eqref{eq1.8}--\eqref{eq1.9}}. If $v=0$ at $\{\bm x_0\}\times(0,T),$ then $g\equiv0$ in $(0,T)$.
\end{lemma}

Let Lemma~\ref{lem4.1} be proved. Then we can complete the proof of Theorem~\ref{thm1} as follows. Let $u$ be the solution to \eqref{eq1.4}--\eqref{eq1.5} and satisfy
\begin{equation}\label{eq4.2}
u=0\quad\mbox{in }\om\times(0,T)\qquad\mbox{or}\qquad u=0\quad\mbox{at }\{\bm x_0\}\times(0,T).
\end{equation}
Setting $v:=(J^\be)'u$ and $g:=(J^\be)'\mu\in L^2(0,T)$, we have \eqref{eq2.1}--\eqref{eq2.2} and $v\in{}_\al H(0,T;L^2(\Om))$ according to Proposition \ref{prop2.1}. Moreover, \eqref{eq4.2} yields $v=0$ in $\om\times(0,T)$ or $v=0$ at $\{\bm x_0\}\times(0,T)$. Thus, since $(J^\be)'\mu=0$ in $L^2(0,T)$ implies $\mu=0$ in $^{-\be}H(0,T)$, the application of Lemma~\ref{lem4.1} completes the proof of Theorem~\ref{thm1}. Thus it suffices to prove Lemma~\ref{lem4.1}.

\begin{proof}[Proof of Lemma \ref{lem4.1}]
We set $w:=J_{1-\al}v$. We note that by the injectivity of $J_{1-\al}$, it follows that $v=0$ in $\om\times(0,T)$ and $v=0$ at $\{\bm x_0\}\times(0,T)$ are equivalent to $w=0$ in $\om\times(0,T)$ and $w=0$ at $\{\bm x_0\}\times(0,T)$, respectively. Hence, we can reduce the proof to the following:

{\it Let $w$ satisfy \eqref{eq3.3}.  Then

{\rm(i)} Let $g\not\equiv0$ in $(0,T)$. If $w=0$ in $\om\times(0,T),$ then $f\equiv0$ in $\Om$.

{\rm(ii)} Let $f\in L^2(\Om)$ satisfy {\rm\eqref{eq1.8}--\eqref{eq1.9}}. If $w=0$ at $\{\bm x_0\}\times(0,T),$ then $g\equiv0$ in $(0,T)$.}\medskip

{\it Proof of (i).} The proof is similar to that of Jiang, Li, Liu and Yamamoto \cite[Theorem 2.6]{JLLY} and we describe the essence.

In terms of Lemma~\ref{lem3.1}, we have
\[
w(\,\cdot\,,t)=\int^t_0g(s)z(\,\cdot\,,t-s)\,\rd s=0\quad\mbox{in }\om,\ 0<t<T.
\]
The Titchmarsh convolution theorem (e.g., \cite{Ti}) yields that there exists $t_*\in[0,T]$ such that $g=0$ in $(0,T-t_*)$ and $z=0$ in $\om\times(0,t_*)$. Since $g\not\equiv0$, we see that $t_*>0$, indicating that $z(\,\cdot\,,t)=0$ in $\om$ holds for $t$ in some open interval in $\BR$. We apply a uniqueness result (e.g., \cite{JLLY} for non-symmetric $A$) to obtain $z\equiv0$ in $\Om\times(0,\infty)$.  Consequently \eqref{eq3.1} implies $\pa_t^\al(z-f)\equiv0$ and thus $z-f=J_\al\pa_t^\al(z-f)\equiv0$ in $\Om\times(0,T)$. By $z\equiv0$ in $\Om\times(0,T)$, we reach $f\equiv0$ in $\Om$. This completes the proof of Lemma~\ref{lem4.1}(i).\medskip

{\it Proof of (ii).} Henceforth, we set $(f,g):=\int_\Om f(\bm x)g(\bm x)\,\rd\bm x$.

We define $A$ by \eqref{eq1.1} with $b_j=0$ ($j=1,\ldots,d$), $c\ge0$ in $\Om$ and $\cD(A)=H^2(\Om)\cap H^1_0(\Om)$. Then we number all of its eigenvalues with their multiplicities as
\[
0<\la_1\le\la_2\le\cdots\longrightarrow\infty.
\]
Let $\vp_n$ be an eigenfunction for $\la_n$: $A\vp_n=\la_n\vp_n$ such that 
$\{\vp_n\}_{n\in\BN}$ forms an orthonormal basis in $L^2(\Om)$.

Then, the fractional power $A^\ga$ is defined with $\ga>0$, and
\[
\cD(A^\ga)\subset H^{2\ga}(\Om),\quad\|f\|_{H^{2\ga}(\Om)}\le C\left(\sum_{n=1}^\infty\la_n^{2\ga}|(f,\vp_n)|^2\right)^{1/2}
\]
for $\ga\ge0$ (e.g., Fujiwara \cite{Fu}, Pazy \cite{Pa}). Moreover, we define the Mittag-Leffler functions $E_{\al,\be}(z)$ with $\al,\be>0$ by 
\[
E_{\al,\be}(z)=\sum_{k=0}^\infty\f{z^k}{\Ga(\al k+\be)},\quad z\in\BC,
\]
where the power series is uniformly and absolutely convergent in any compact set in $\BC$ (e.g., Gorenflo, Kilbas, Mainardi and Rogosin \cite{GKMR}, Podlubny \cite{Po}). Then we can represent 
\begin{equation}\label{eq4.6}
z(\,\cdot\,,t)=\sum_{n=1}^\infty E_{\al,1}(-\la_n t^\al)(f,\vp_n)\vp_n\quad\mbox{in }L^2(0,T;L^2(\Om)).
\end{equation}
Then by \eqref{eq1.9} we can prove that for any fixed $\de>0$ and $T_1>0$, 
\begin{equation}\label{eq4.7}
\mbox{the series \eqref{eq4.6} is absolutely convergent in }L^\infty(\de,T_1;C(\ov\Om))
\end{equation}
and
\begin{equation}\label{eq4.8}
z(\bm x_0,\,\cdot\,)\in L^1(0,T_1).
\end{equation}

{\it Verification of \eqref{eq4.7} and \eqref{eq4.8}.} First let $d=1,2,3$. Since 
\[
|E_{\al,1}(-\la_n t^\al)|\le\f C{1+\la_n t^\al},\quad n\in\BN,\ t>0
\]
(e.g., \cite{Po}), using \eqref{eq4.6}, we have
\begin{align*}
\|A z(\,\cdot\,,t)\|^2_{L^2(\Om)} & =\sum_{n=1}^\infty\la_n^2|E_{\al,1}(-\la_n t^\al)|^2|(f,\vp_n)|^2\le C\sum_{n=1}^\infty\left|\f{\la_n}{1+\la_n t^\al}\right|^2|(f,\vp_n)|^2\\
& \le C\,t^{-2\al}\sum_{n=1}^\infty\left|\f{\la_n t^\al}{1+\la_n t^\al}\right|^2|(f,\vp_n)|^2\le C\,t^{-2\al}\sum_{n=1}^\infty|(f,\vp_n)|^2\\
& =C\,t^{-2\al}\|f\|^2_{L^2(\Om)},
\end{align*}
that is, $\|z(\,\cdot\,,t)\|_{C(\ov\Om)}\le C\,t^{-\al}\|f\|_{L^2(\Om)}$ because
\[
\|z(\,\cdot\,,t)\|_{C(\ov\Om)}\le C\|z(\,\cdot\,,t)\|_{H^2(\Om)}\le C'\|A z(\,\cdot\,,t)\|_{L^2(\Om)}
\]
which is seen by $d\le3$ and the Sobolev embedding. Therefore, \eqref{eq4.7} and \eqref{eq4.8} are seen for $d=1,2,3$.

Next let $d\ge4$. Then, since $\te>d/4-1$ by \eqref{eq1.9}, the Sobolev embedding yields $\cD(A^{1+\te})\subset H^{2+2\te}(\Om)\subset C(\ov\Om)$. Therefore,
\begin{align*}
|z(\bm x_0,t)| & \le C\|z(\,\cdot\,,t)\|_{C(\ov\Om)}\le C\|z(\,\cdot\,,t)\|_{H^{2+2\te}(\Om)}\le C\|A^{1+\te}z(\,\cdot\,,t)\|_{L^2(\Om)}\\
& =C\left\|\sum_{n=1}^\infty\la_n E_{\al,1}(-\la_n t^\al)(A^\te f,\vp_n)\vp_n\right\|_{L^2(\Om)}\\
& =C\left(\sum_{n=1}^\infty\la_n^2|E_{\al,1}(-\la_n t^\al)|^2|(A^\te f,\vp_n)|^2\right)^{1/2}\\
& \le C\,t^{-\al}\left(\sum_{n=1}^\infty\left|\f{\la_n t^\al}{1+\la_n t^\al}\right|^2|(A^\te f,\vp_n)|^2\right)^{1/2}\le C\,t^{-\al}\|A^\te f\|_{L^2(\Om)},
\end{align*}
so that the verification of \eqref{eq4.7} and \eqref{eq4.8} is complete.\medskip

Similarly to the proof of (i), if $g\not\equiv0$ in $(0,T)$, then $z(\bm x_0,\,\cdot\,)\equiv0$ in $(0,t_*)$ with some constant $t_*>0$. In terms of \eqref{eq4.7}, we apply the $t$-analyticity of $z(\bm x_0,t)$ (e.g., Sakamoto and Yamamoto \cite{SY}) to reach $z(\bm x_0,t)=0$ for all $t>0$. Therefore, we obtain
\[
\sum_{n=1}^\infty E_{\al,1}(-\la_n t^\al)(f,\vp_n)\vp_n(\bm x_0)=0,\quad t>\de,
\]
where the series is absolutely convergent in $L^\infty(\de,T_1)$ with arbitrary $T_1>0$.

Not counting the multiplicities, we rearrange all the eigenvalues of $A$ as
\[
0<\rho_1<\rho_2<\cdots\longrightarrow\infty
\]
and by $\{\vp_{n j}\}_{1\le j\le d_n}$ we denote an orthonormal basis of $\ker(\rho_n-A)$. In other words, $\{\rho_n\}_{n\in\BN}$ is the set of all distinct eigenvalues of $A$. We set $P_n f:=\sum_{j=1}^{d_n}(f,\vp_{n j})\vp_{n j}$. Hence we can write
\begin{equation}\label{eq4.10}
\sum_{n=1}^\infty E_{\al,1}(-\rho_n t^\al)(P_n f)(\bm x_0)=0,\quad t>\de.
\end{equation}
On the other hand, we know
\begin{equation}\label{eq4.11}
E_{\al,1}(-\rho_n t^\al)=\f1{\Ga(1-\al)\rho_n t^\al}+O\left(\f1{\rho_n^2t^{2\al}}\right)\quad\mbox{as }t\to\infty
\end{equation}
(e.g., \cite[Theorem 1.4 (pp.33--34)]{Po}).

Since the series in \eqref{eq4.10} converges in $L^\infty(\de,T_1)$, extracting a subsequence of partial sums for the limit, we see that the subsequence of the partial sums is convergent for almost all $t\in(\de,T_1)$. Hence, by \eqref{eq4.11} we obtain
\[
\f1{\Ga(1-\al)}\f1{t^\al}\sum_{n=1}^\infty\f{(P_n f)(\bm x_0)}{\rho_n}+\f1{t^{2\al}}\sum_{n=1}^\infty c_n(P_n f)(\bm x_0)=0
\]
for almost all $t\in(\de,\infty)$, where $|c_n|\le O(\rho_n^{-2})$. Multiplying by $t^\al$ and choosing a sequence $\{t_m\}$ tending to $\infty$, we reach
\[
\sum_{n=1}^\infty\f{(P_n f)(\bm x_0)}{\rho_n}=0.
\]
Since it is assumed in Theorem~\ref{thm1}(ii) that $c\ge 0$ in $\Om$, we see that $A^{-1}$ exists and is a bounded operator from $L^2(\Om)$ to itself and 
\[
A^{-1}f=\sum_{n=1}^\infty\f{P_n f}{\rho_n}\quad\mbox{in }L^2(\Om).
\]
Therefore, we conclude $(A^{-1}f)(\bm x_0)=0$.

Set $\psi:=A^{-1}f\in H^2(\Om)\cap H^1_0(\Om)$. Then $A\psi=f$ in $\Om$ and
\begin{equation}\label{eq4.12}
\psi(\bm x_0)=0.
\end{equation}
Since $c\ge0$ in \eqref{eq1.1} and $f\ge0$ or $f\le0$ in $\Om$, using \eqref{eq4.12} and applying the strong maximum principle for $A$ (e.g., Gilbarg and Trudinger \cite{GT}), we conclude $\psi\equiv0$ in $\Om$, that is, $f\equiv0$ in $\Om$. This contradicts the assumption $f\not\equiv0$ in $\Om$. Therefore, $t_*>0$ is impossible. Hence $t_*=0$ and so Titchmarsh convolution theorem yields $g\equiv0$ in $(0,T)$. This completes the proof of Lemma~\ref{lem4.1}(ii).\qed
\end{proof}


\section{Concluding remarks}\label{sec-remark}

In this article, we consider an initial-boundary value problem for
\begin{equation}\label{eq5.1}
(\pa_t^\al+A)u=\mu(t)f(\bm x),
\end{equation}
where $\mu=\mu(t)$ is in a Sobolev space of negative order. The main machinery is to operate the extended Riemann-Liouville fractional integral operator $(J^\be)'$ (see Lemma~\ref{lem1.2} in Sect.~\ref{sec-premain}) to reduce \eqref{eq5.1} to
\begin{equation}\label{eq5.2}
(\pa_t^\al+A)v=((J^\be)'\mu)(t)f(\bm x),
\end{equation}
where $v:=(J^\be)'u$ and $(J^\be)'\mu\in L^2(0,T)$. Thus for the inverse source problems for \eqref{eq5.1}, we can assume that $\mu\in L^2(0,T)$ by replacing \eqref{eq5.1} by \eqref{eq5.2}.

In this article, we limit the range of $\be$ to $(0,1)$ for simplicity, but we can choose arbitrary $\be>0$. Therefore, for inverse source problems for \eqref{eq5.1}, it is even sufficient to assume that $\mu$ is smooth or $\mu\in L^\infty(0,T)$.

The same transformation for inverse source problems is valid for general time-fractional differential equations including fractional derivatives of variable orders $\al_k(\bm x)$:
\begin{equation}\label{eq5.3}
\sum_{k=1}^N p_k(\bm x)\pa_t^{\al_k(\bm x)}u+A u=\mu(t)f(\bm x)
\end{equation}
with suitable conditions on $p_k,\al_k$. In particular, also for \eqref{eq5.3}, we can similarly discuss the determination of $\mu\in{}^{-\be}H(0,T)$ with $\be>0$ by transforming $\mu$ to a smooth function.


\appendix
\section{Proof of \eqref{eq1.10}}\label{sec-app}

Since $g:=(J^\be)'\mu\in L^2(0,T)$ by $\mu\in{}^{-\be}H(0,T)$, by Lemma~\ref{lem1.2}(ii), it suffices to prove the solution $v:=(J^\be)'u\in L^2(0,T;C(\ov\Om))$ to \eqref{eq2.2} if $f$ satisfies \eqref{eq1.9}. By \cite{SY} for example, we have
\[
v(\,\cdot\,,t)=\sum_{n=1}^\infty\left(\int^t_0s^{\al-1}E_{\al,\al}(-\la_n s^\al)g(t-s)\,\rd s\right)(f,\vp_n)\vp_n,\quad0<t<T.
\]
Moreover, we know
\begin{equation}\label{eq6.2}
\la_n t^{\al-1}E_{\al,\al}(-\la_n t^\al)=-\f\rd{\rd t}E_{\al,1}(-\la_n t^\al),\quad t>0
\end{equation}
and
\begin{equation}\label{eq6.3}
t^{\al-1}E_{\al,\al}(-\la_n t^\al)\ge0,\quad t>0.
\end{equation}
We can directly verify \eqref{eq6.2} by the termwise differentiation because $E_{\al,1}(-\la_n t^\al)$ is an entire function, while \eqref{eq6.3} follows from the complete monotonicity of $E_{\al,1}(-\la_n t^\al)$ (e.g., Gorenflo, Kilbas, Mainardi and Rogosin \cite{GKMR}).

Let $d=1,2,3$. Then
\[
\|A v(\,\cdot\,,t)\|_{L^2(\Om)}^2=\sum_{n=1}^\infty\left|\int^t_0\la_n s^{\al-1} E_{\al,\al}(-\la_n s^\al)g(t-s)\,\rd s\right|^2|(f,\vp_n)|^2.
\]
Hence, Young's convolution inequality implies 
\begin{align}
& \quad\,\left\| 
\int^t_0\la_n s^{\al-1}E_{\al,\al}(-\la_n s^\al)g(t-s)\,\rd s\right\|_{L^2(0,T)}\nonumber\\
& \le\left\|\la_n t^{\al-1}E_{\al,\al}(-\la_n t^\al)\right\|_{L^1(0,T)}\|g\|_{L^2(0,T)},\label{eq6.4}
\end{align}
and \eqref{eq6.2} and \eqref{eq6.3} yield
\begin{align}
\left\|\la_n t^{\al-1}E_{\al,\al}(-\la_n t^\al)\right\|_{L^1(0,T)} & =\int^T_0\la_n t^{\al-1}E_{\al,\al}(-\la_n t^\al)\,\rd t\nonumber\\
& =-\int^T_0\f\rd{\rd t}E_{\al,1}(-\la_n t^\al)\,\rd t\nonumber\\
& =1-E_{\al,1}(-\la_n T^\al)\le1.\label{eq6.5}
\end{align}
Therefore, by \eqref{eq6.4} we see
\begin{align*}
& \quad\,\int^T_0\|A v(\,\cdot\,,t)\|^2_{L^2(\Om)}\,\rd t\\
& =\sum_{n=1}^\infty\left\|\int^t_0\la_n s^{\al-1}E_{\al,\al}(-\la_n s^\al)g(t-s)\,\rd s\right\|^2_{L^2(0,T)}|(f,\vp_n)|^2\\
& \le\|g\|^2_{L^2(0,T)}\sum_{n=1}^\infty|(f,\vp_n)|^2\le\|g\|^2_{L^2(0,T)}\|f\|^2_{L^2(\Om)}.
\end{align*}
Therefore, the Sobolev embedding $H^2(\Om)\subset C(\ov\Om)$ by $d=1,2,3$, yields 
\[
\|v\|^2_{L^2(0,T;C(\ov\Om))}\le C\|g\|^2_{L^2(0,T)}\|f\|^2_{L^2(\Om)},
\]
which means \eqref{eq1.10} for $d=1,2,3$.

Let $d\ge4$. Then we assume $f\in\cD(A^\te)$ with $\te>d/4-1$. Since $\la_n^\te\vp_n=A^\te\vp_n$, applying \eqref{eq6.4} and \eqref{eq6.5}, we have
\[
\|A^{1+\te}v(\,\cdot\,,t)\|^2_{L^2(\Om)}=\sum_{n=1}^\infty\left|\int^t_0\la_n s^{\al-1}E_{\al,\al}(-\la_n s^\al)g(t-s)\,\rd s\right|^2|(A^\te f,\vp_n)|^2,
\]
and so
\[
\|v\|_{L^2(0,T;\cD(A^{1+\te}))}^2\le C\sum_{n=1}^\infty\|g\|^2_{L^2(0,T)}|(A^\te f,\vp_n)|^2\le C\|g\|^2_{L^2(0,T)}\|A^\te f\|^2_{L^2(\Om)}.
\]
By the Sobolev embedding $\cD(A^{1+\te})\subset H^{2+2\te}(\Om)\subset C(\ov\Om)$, we complete the proof of \eqref{eq1.10}.\bigskip


{\bf Acknowledgement}\ \ 
Y.\! Liu is supported by Grant-in-Aid for Early Career Scientists 20K14355 and 22K13954, JSPS.
M.\! Yamamoto is supported by Grant-in-Aid for Scientific Research (A) 20H00117 and Grant-in-Aid for Challenging Research (Pioneering) 21K18142, JSPS.
%
%

\end{document}